\documentclass[a4paper,reqno]{amsart}
\usepackage{amssymb}
\usepackage{array,multirow}
\usepackage{hyperref}
\usepackage{amsmath} 
\usepackage{pifont}
\usepackage[mathb]{mathabx}
\usepackage[usenames]{color}
\usepackage{amsfonts}
\usepackage[all]{xypic}
\usepackage{graphicx}
\usepackage{psfrag}
\usepackage{verbatim}
\usepackage{pstricks}
\usepackage[latin1]{inputenc}
\usepackage{mathrsfs}
\usepackage[all,knot]{xy}
\usepackage{amssymb, amsthm, amsmath, pifont} 
\usepackage{enumerate}
\usepackage{pdfpages}

\usepackage{caption}
\usepackage{subcaption}
\usepackage{enumitem}

\newtheorem{thm}{Theorem}
\newtheorem{corollary}[thm]{Corollary}

\newtheorem{example}[thm]{Example}
\newtheorem{definition}[thm]{Definition}
\newtheorem{remark}[thm]{Remark}
\newtheorem{prop}[thm]{Proposition}
\newtheorem{claim}[thm]{Claim}
\newtheorem{observation}[thm]{Observation}

\newtheorem*{question}{Question}

\newenvironment{Remark}{\begin{remark}\rm}{\end{remark}}

\def\et{\quad\mbox{and}\quad}

\def\epsilon{\varepsilon}

\def\s{{\sigma}}

\def\alt{{\rm alt}}
\def\Kh{{\rm Kh}}
\def\dalt{{\rm dalt}}
\def\thick{{\rm width}}

\begin{document}
\date{\today}
\address{Boston College, Department of Mathematics, Maloney Hall, Chestnut Hill, MA 02467, United States}
\address{Mathematisches Institut, Universit\"at zu K\"oln, Weyertal 86-90, 50931 K\"oln, Germany}
\address{Fakult\"at f\"ur Mathematik, Universit\"at Regensburg, 93040 Regensburg, Germany}

\author{Peter Feller}
\author{Simon Pohlmann}
\author{Raphael Zentner}

\email{peter.feller.2@bc.edu}
\email{simonpohlmann@freenet.de}
\email{raphael.zentner@mathematik.uni-regensburg.de}
\thanks{The first author gratefully acknowledges support by the Swiss National Science Foundation Grant 155477: \emph{Positive Braids and Deformations}. The third author gratefully acknowledges support by the SFB `Higher Invariants' at the University of Regensburg, funded by the Deutsche Forschungsgesellschaft (DFG)}
\title{Alternating numbers of torus knots with small braid index}
\begin{abstract} We calculate the alternating number of torus knots with braid index 4 and less. For the lower bound, we use the 
upsilon-invariant recently introduced by Ozsv\'ath, Stipsicz, and Szab\'o. For the upper bound, we use a known bound for braid index $3$ and a new bound for braid index $4$. Both bounds coincide, so that we obtain a sharp result.
\end{abstract}
\maketitle
\section{Introduction}
Kawauchi introduced the \emph{alternating number} $\alt(K)$ of a knot $K$\textemdash the minimal number of crossing changes needed to turn a diagram of $K$ into the diagram of an alternating knot~\cite{Kawauchi_10_OnAltNrOfLinks}.
Our main result determines the alternating number for all torus knots with braid index 4 or less.
\begin{thm}\label{thm:main}If $K$ is a torus knot of braid index 3 or 4,
then $\alt(K)=\lfloor \frac{1}{3}g(K)\rfloor$.

In other words,
for all positive integers $n$, we have
\[\alt(T_{3,3n+1})=\alt(T_{3,3n+2})=\alt(T_{4,2n+1})=n.\]
\end{thm}
The proof of Theorem~\ref{thm:main} consists of two parts. We use Ozsv\'ath, Stipsicz, and Szab\'o's
$\Upsilon$-invariant~\cite{OSS_2014} to improve previously known lower bounds for the alternating number.
The necessary upper bounds are provided by an explicit geometric construction in the case of braid index 4, and by Kanenobu's bound of \cite{Kanenobu_10_UpperBoundForAltOfTorusKnots}.
\\

Let us put Theorem~\ref{thm:main} in context. Torus knots with braid index $2$ are alternating; in other words, their alternating number is zero.
For torus knots with braid index 3, our result is a slight improvement on previous work of Kanenobu. In \cite{Kanenobu_10_UpperBoundForAltOfTorusKnots}, he established that
\[\alt(T_{3,3n+1})=\alt(T_{3,3n+2})=n\] for even positive integers $n$, whereas for odd integers $n$ he is left with the ambiguity that
\[\alt(T_{3,3n+1}),\alt(T_{3,3n+2})\in \{n-1,n\} .
\]

For torus knots of braid index $4$, Kanenobu established
\[n \leq \alt(T_{4,2n+1})\leq \frac{3}{2}n\et n-1\leq \alt(T_{4,2n+1})\leq \frac{3}{2}n-\frac{1}{2}\] for even and odd $n$, respectively~\cite{Kanenobu_10_UpperBoundForAltOfTorusKnots}. Therefore, Theorem~\ref{thm:main} improves both the previuosly existing lower and upper bound.
\\

A related knot invariant is the dealternating number $\dalt(K)$ of a knot. This number is the minimal number of crossing changes that one needs for turning a diagram of $K$ into an alternating {\em diagram}.
Clearly,
\[\alt(K)\leq\dalt(K)\] for all knots $K$. The dealternating number might appear less appealing at first sight. However, there exists the following interesting connection to quantum topology, due to Asaeda and Przytycki (reproved by Champanerkar-Kofman in \cite{Champanerkar-Kofman} with a spanning tree model for Khovanov homology): for all knots $K$,
\begin{equation}\label{eq:lowerboundKh}
\thick(\Kh(K))-2\leq\dalt(K),
\end{equation}
where $\Kh$ denotes the unreduced Khovanov homology~\cite{Asaeda-Przytycki}, and $\thick(\Kh(K))$ denotes the number of $\delta$-diagonals with $\delta$-grading greater or equal the lowest $\delta$-grading on which the Khovanov homology has support and less than or equal the highest $\delta$-grading on which Khovanov homology has support. The inequality~\eqref{eq:lowerboundKh} can be used to show that the alternating number differs from the dealternating number in general. For instance, any Whitehead double $W_K$ of a (non-trivial) knot $K$ has alternating number 1, while $\thick(\Kh(K))$ is in general larger than $3$ for Whitehead doubles.
\\

Using Turner's calculation of $\thick(\Kh)$ for torus knots of braid index three~\cite{Turner_08_ASpecSequForKh}, Abe and Kishimoto used inequality~\eqref{eq:lowerboundKh} to calculate the dealternating number for torus knots with braid index $3$.
However, the width $\thick(\Kh)$ is unknown for torus knots of braid index $4$. In fact, by 
work of Beheddi, one has $n+2\leq \thick(\Kh(T_{4,2n+1}))$, see\cite{Benheddi_PhD}, and, conjecturally, this is an equality.
\begin{question}
Does Theorem~\ref{thm:main} also hold for the dealternating number? In other words, are there geometric constructions similar to the ones provided below, that show $\dalt(T_{4,2n+1})=\alt(T_{4,2n+1}) =n$?
\end{question}
A positive answer would determine $\thick(\Kh(T_{4,2n+1}))$ to be $n+2$. This was part of the original motivation for the study conducted in this paper. However, it is impossible to immediately use the constructions for $\alt(T_{4,2n+1})\leq n$ presented in Section~\ref{sec:upperbunds} to show $\dalt(T_{4,2n+1})\leq n$; compare Remark~\ref{rem:daltisnotobvious}.
\\


\section{Lower bounds for the alternating number}\label{sec:lowerbounds}
In~\cite{Abe_09_EstOfAltOfTorusKnot}, Abe observed that
\[\frac{|s(K)-\s(K)|}{2}\leq\alt(K)\]
for all knots $K$, where $s$ and $\s$ denote Rasmussen's invariant~\cite{rasmussen_sInv} and Trotter's signature~\cite{Trotter_62_HomologywithApptoKnotTheory}, respectively. In fact, this lower bound works similarly with other knot invariants:

\begin{prop}\label{obs:lowerbounds}
Let $\psi_1$ and $\psi_2$ be any real-valued knot invariants such that
\begin{enumerate}[label=(\roman*)]
\item for all alternating knots
$\psi_1$ and $\psi_2$ are equal and
\item\label{eq:crossingchange}
if $K_+$ and $K_-$ are two knots such that $K_-$ is obtained from $K_+$ by changing a positive crossing to a negative crossing, then
\[\psi_i({K_-})-1 \leq \psi_i(K_+)\leq \psi_i({K_-})\] for $i=1,2$.
\end{enumerate}
Then for all knots $K$, we have
\[|\psi_1(K)-\psi_2(K)|\leq \alt(K).\]
\end{prop}
\begin{proof}
For $i= 0, \dots, n$, let $K_i$ be a sequence of knots such that for $i=1, \dots, n-1$ the knot $K_{i+1}$ results from $K_i$ through a crossing change, and such that $K_0$ is alternating. Induction on $n$ shows that the difference $|\psi_1(K_n) - \psi_2(K_n)|$ can be at most $n$.
\end{proof}
For Ozsv\'ath and Szab\'o's $\tau$-invariant the negative $\psi_1(K)= -\tau(K)$ satisfies~\ref{eq:crossingchange} from Proposition~\ref{obs:lowerbounds}; see~\cite{OzsvathSzabo_03_KFHandthefourballgenus}.
Similarly, the invariant $\psi_2(K) = \Upsilon_K(1) = \upsilon(K)$ does satisfy~\ref{eq:crossingchange}, and
\[\psi_1(A)= -\tau(A)= \Upsilon_A(1) = \upsilon(A)=\psi_2(A) \]
for all alternating knots $A$.
Here $\Upsilon_K(t)$ (denoted by $\upsilon(K)$ when $t=1$) is the real valued knot-invariant (depending piecewise-linearly on a parameter $t$ in $[0,2]$) introduced by Ozsv\'ath, Stipsicz, and Szab\'o~\cite{OSS_2014}. Therefore, we get the following.

\begin{corollary}\label{cor:lowerbound}
 For all knots $K$, 
 we have
  \[
  	\left|\tau(K) + \upsilon(K) \right| \leq \alt(K) \, .
  \]
\end{corollary}

We note that other invariants rather than $\tau$ can be used and will yield the same lower bounds for the alteranting  number on torus knots; for example, Rasmussen's $s$-invariant or any concordance invariant with the properties described in~\cite[Theorem~1]{Livingston_Comp}. The  $\tau$-invariant seems to be the canonical choice to work with since $\Upsilon$ is a generalization of it: indeed, one has $-\tau=\lim_{t
{\to}0}
\frac{\Upsilon(t)}{t}$; see \cite[Proposition~1.6]{OSS_2014}.

\begin{prop}\label{lower bound}
For all positive integers $n$, we have the following bounds for the alternating number.
\begin{align*}
n \leq \alt(T_{3,3n+1}),\quad n  \leq\alt(T_{3,3n+2}), \et n  \leq \alt(T_{4,2n+1}).
\end{align*}
\end{prop}
\begin{proof}
This is immediate from calculating $|\tau+\upsilon|$ for the involved knots.
On positive torus knots $\tau$ equals the three-genus:
\[\tau(T_{p,q})=\frac{(p-1)(q-1)}{2},\] for all coprime positive integers $p$ and $q$; see~\cite[Corollary~1.7]{OzsvathSzabo_03_KFHandthefourballgenus}.
For torus knots (and more generally $L$-space knots) Ozsv\'ath, Stipsicz, and Szab\'o~\cite[Theorem~1.15]{OSS_2014} provided a procedure to calculate $\Upsilon(t)$ from the Alexander polynomial. With this procedure one calculates
\[\upsilon(T_{3,3n+1})=-2n=\upsilon(T_{4,2n+1})\et \upsilon(T_{3,3n+1})=-2n-1,\] for all $n$; compare~\cite[Proposition~28]{Feller_15_MinCobBetweenTorusknots}, where this tedious but elementary calculation is provided.
The values for $\tau$ and $\upsilon$ combined yield
\begin{align*}
|\tau(T_{3,3n+1})+\upsilon(T_{3,3n+1})|&=3n-2n=n\\
|\tau(T_{3,3n+2})+\upsilon(T_{3,3n+2})|&=3n+1-2n-1=n\\
|\tau(T_{4,2n+1})+\upsilon(T_{4,2n+1})|&=3n-2n=n.
\end{align*}
This concludes the proof since $|\tau+\upsilon|$ is a lower bound for the alternating number by Corollary~\ref{cor:lowerbound}.
\end{proof}
\section{Upper bounds for the alternating number}\label{sec:upperbunds}
For torus knots with braid index $3$, upper bounds for the alternating number where calculated by Kanenobu~\cite{Kanenobu_10_UpperBoundForAltOfTorusKnots}; compare also~\cite{Feller_14_GordianAdjacency}, where this is recovered from a different perspective. Abe and Kishimoto showed that the same upper bounds hold for the dealternating number~\cite{AbeKishimoto_Dealtof3braids}.
\begin{prop}[{\cite[Theorem~8]{Kanenobu_10_UpperBoundForAltOfTorusKnots}},{\cite[Theorem~2.5]{AbeKishimoto_Dealtof3braids}}]
\label{abe-kishimoto}
For all positive integers $n$,
\[\alt(T_{3,3n+1}) \leq \dalt(T_{3,3n+1})\leq n \et \alt(T_{3,3n+2}) \leq \dalt(T_{3,3n+2})\leq n.\]
\end{prop}

We provide new upper bounds for torus knots of braid index 4.

\begin{prop}\label{upper bound}
	Let $n \geq 2$ be an integer. There is a diagram of the torus knot $T_{4,2n+1}$ such that $n$ crossing changes yield the knot $T_{2,2n+1} \, \# \, T_{2,2n+1}$. In particular,
	\[ \alt(T_{4,2n+1}) \leq n . \]
\end{prop}

\begin{Remark}\label{rem:toruslinkupperbound}
	Similarly, one can show that for $n\geq 2$ there is a diagram of the torus link $T_{4,2n}$ such that there are $n$ crossing changes which turn this torus link into an alternating link.
\end{Remark}
\begin{Remark}\label{rem:daltisnotobvious}
It is impossible that the diagram for $T_{2,2n+1} \, \# \, T_{2,2n+1}$ provided by Proposition~\ref{upper bound} is alternating. Indeed, assume towards a contradiction that there is a diagram $D_1$ for the torus knot $T_{4,2n+1}$ such that $n$ crossing changes yield an alternating diagram $D_2$ for the knot $T_{2,2n+1} \, \# \, T_{2,2n+1}$. We may assume that $D_1$ and $D_2$ are reduced diagrams.
Since the minimal crossing number of $T_{4,2n+1}$ is $6n+3$, the diagram $D_1$, and thus also $D_2$, has at least $6n+3$ crossings. 
However,
$T_{2,2n+1} \, \# \, T_{2,2n+1}$ has an alternating diagram with $4n+2$ crossings, which contradicts Tait's conjecture that two reduced alternating diagrams for the same knot have the same number of crossings proven by Kauffman, Murasugi and Thistlethwaite~\cite{Kauffman_87_StateModelsAndTheJonesPolynomial,Murasugi_87,Thistlethwaite_87}.
\end{Remark}
\begin{proof}[Proof of Proposition~\ref{upper bound}]
We think of the torus knots $T_{4,2n+1}$ as closures of braids. Using braid relations respectively an isotopy, we see that a `full twist' can be isotoped according to Figure~\ref{isotopy_full_twist}.

\begin{figure}[h!]
\caption{These are identical braids corresponding to isotopic diagrams relative to the ends. The left hand side is standard, the right hand side desription will be used later on.}
\label{isotopy_full_twist}
\def\svgwidth{0.5\columnwidth}
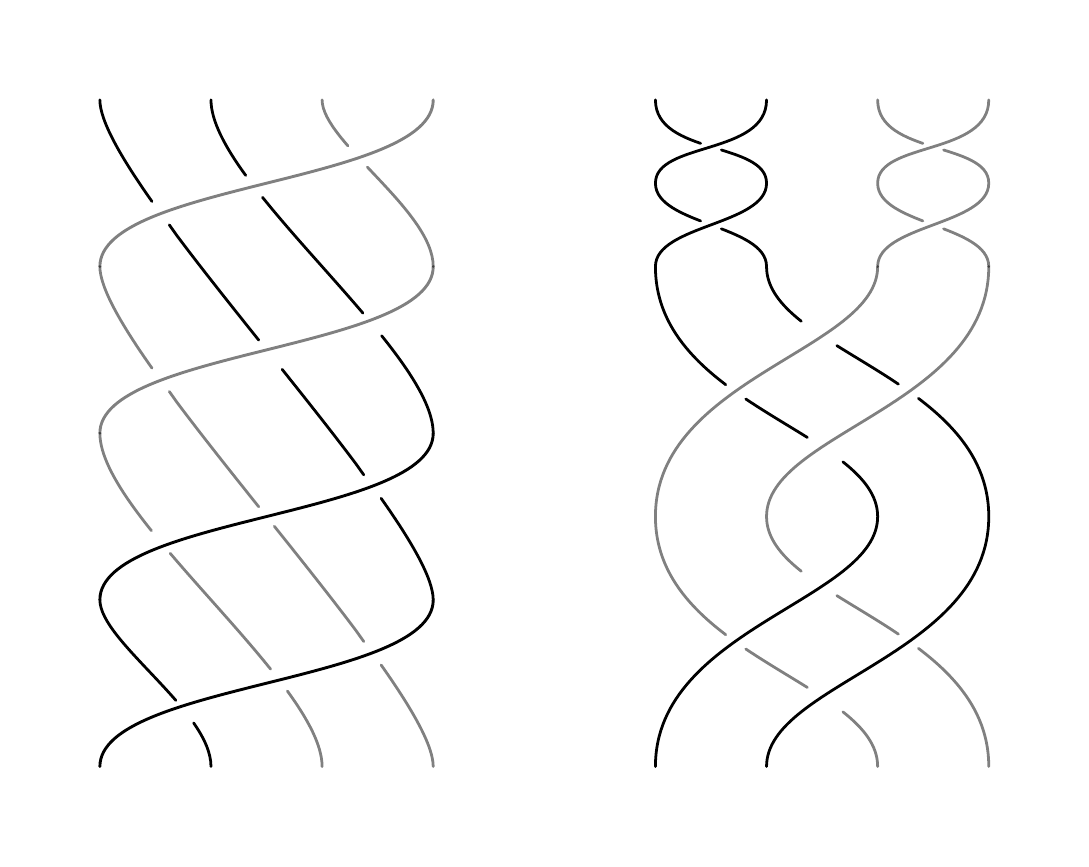
\end{figure}

Similarly, a full and a half twist can be isotoped according to Figure~\ref{isotopy_full_and_halftwist}. We notice a slight asymmetry in the two `bands' in this case.

\begin{figure}[h!]
\caption{Isotopy corresponding to a full and a half twist.}
\label{isotopy_full_and_halftwist}
\def\svgwidth{0.5\columnwidth}
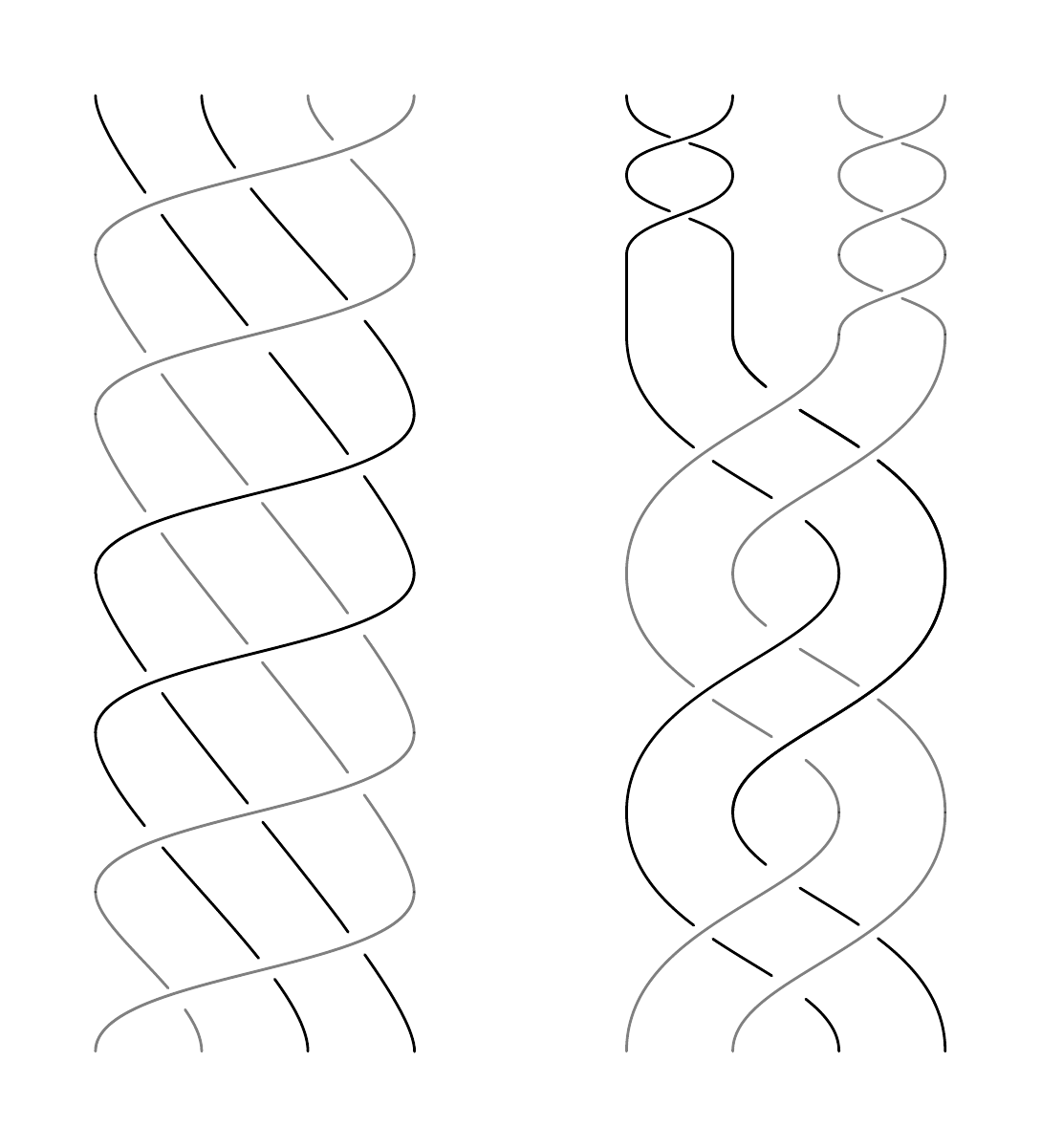
\end{figure}

We observe that these isotopies are compatible with iterations of full twists respectively multiplication of the braids corresponding to full twists. The result will be two bands which, when seen from the top to the bottom, both first twist, and then cross each other as planar bands.

Now in each full twist, we can find two crossing changes in the region where the bands cross with a geometric significance. Figure~\ref{crossingchanges_fulltwist} below shows how we can achieve the two red strands to pass in front of the two green strands. Similarly, Figure~\ref{crossingchanges_full_and_halftwist} shows how we can achieve the two green strands to pass in front of the two red strands.

\begin{figure}[h!]
\centering
\begin{minipage}{.4\textwidth}
  \centering
  \def\svgwidth{0.99\columnwidth}
  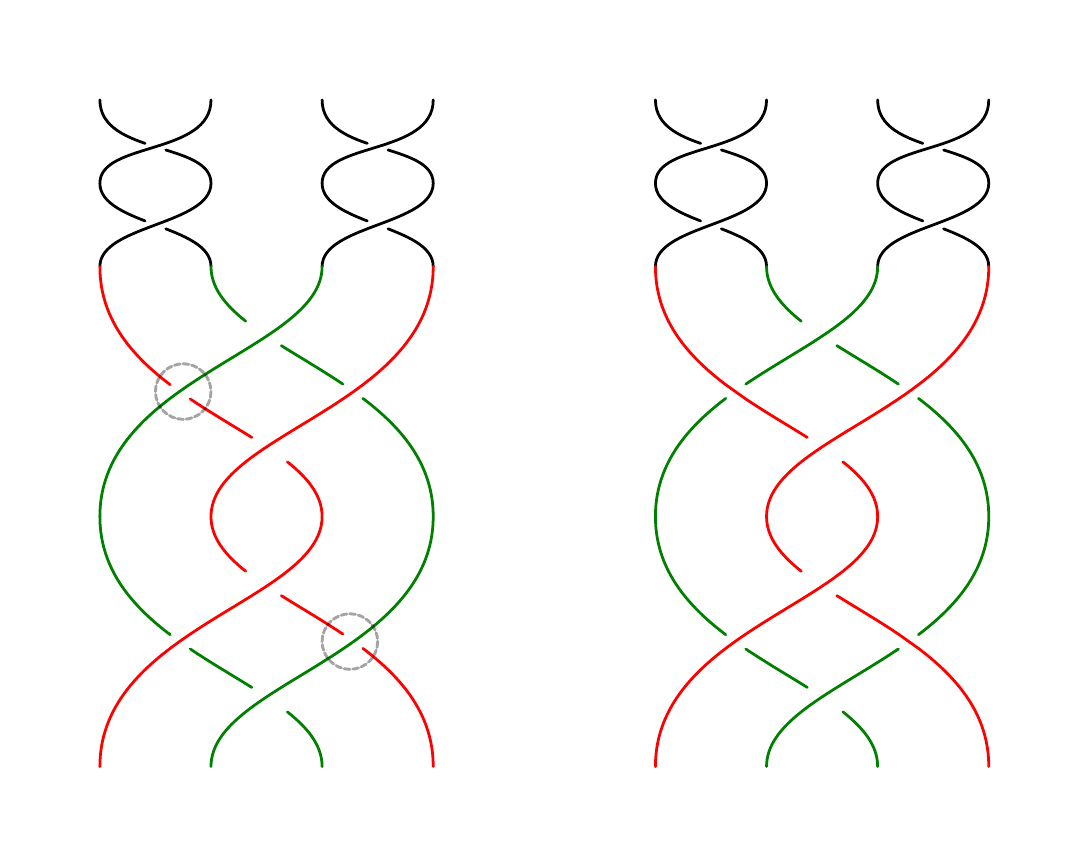
  \captionof{figure}{Two crossing changes bringing the red strands to the front}
  \label{crossingchanges_fulltwist}
\end{minipage}%
\begin{minipage}{.4\textwidth}
  \centering
  \def\svgwidth{0.99\columnwidth}
  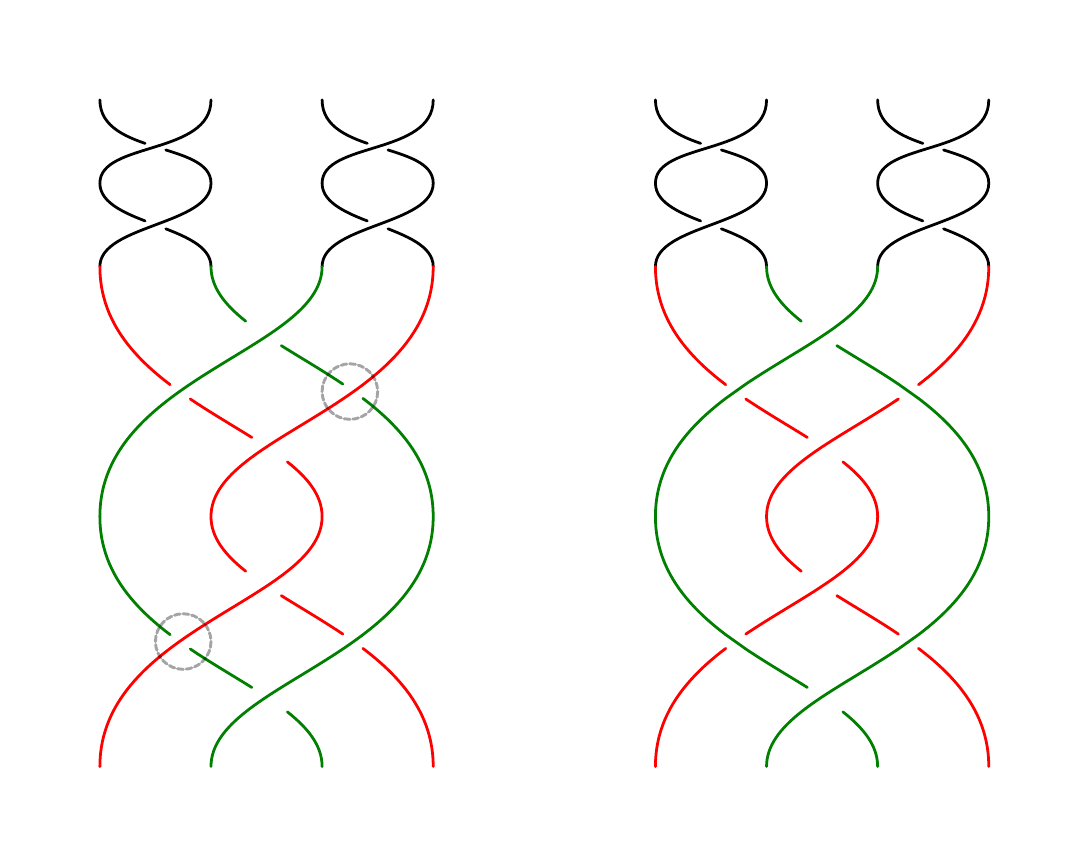
  \captionof{figure}{Two crossing changes bringing the green strands to the front}
  \label{crossingchanges_full_and_halftwist}
\end{minipage}
\end{figure}

Iterating this, we see that with $n$ crossing changes, we transform the braid corresponding to the torus knot $T_{4,2n+1}$ to the braid on the left hand side of Figure~\ref{both braids} if $n$ is even, and to the braid on the right han side if $n$ is odd. In the first case, we have used the crossing changes according to Figure~\ref{crossingchanges_fulltwist}, in the second case we have used those of Figure~\ref{crossingchanges_full_and_halftwist}.

\begin{figure}[h!]
\caption{After $n$ crossing changes, we obtain the braid on the left for $n$ even, and the one on the right for $n$ odd, starting from $T_{4,2n+1}$.}
\label{both braids}
\def\svgwidth{0.4\columnwidth}
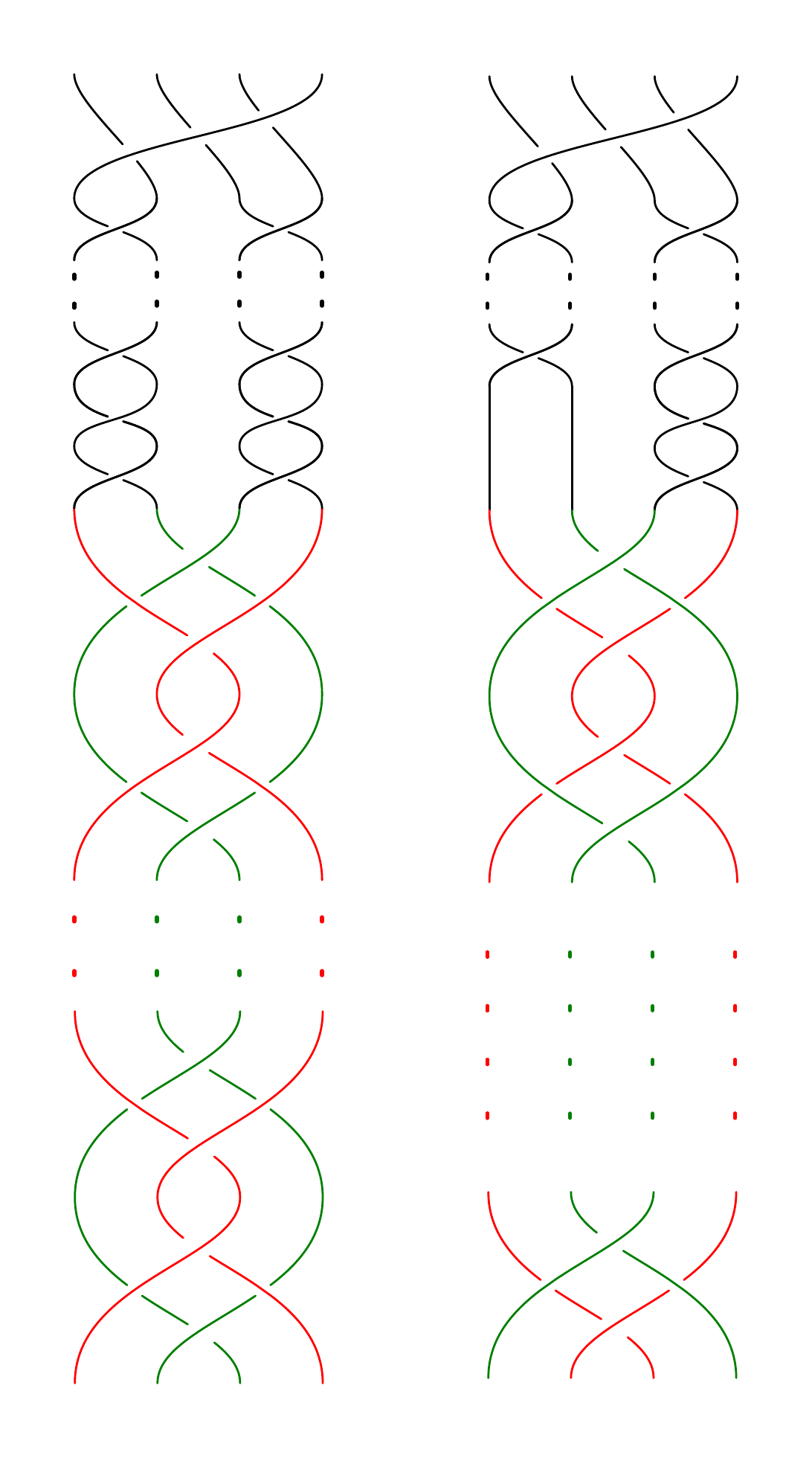
\end{figure}

Finally we observe that the braid closure of this is the connected sum \[T_{2,2n+1} \,  \# \, T_{2,2n+1}.\]
 To see this, we must distinguish the cases $n$ even and $n$ odd. If $n$ is even, we start with the braid closure of the left hand diagram in Figure~\ref{both braids}. We can flip the green strands in the braid closure to the top, passing behind everything else; see Figure~\ref{final isotopy}. Notice that this flipping yields a new crossing between the two flipped strands.

\begin{figure}[h!]
\caption{An isotopy from the braid closure to the connected sum $T_{2,2n+1} \,  \# \, T_{2,2n+1}$. The green strands pass behind everything else.}
\label{final isotopy}
\def\svgwidth{0.8\columnwidth}
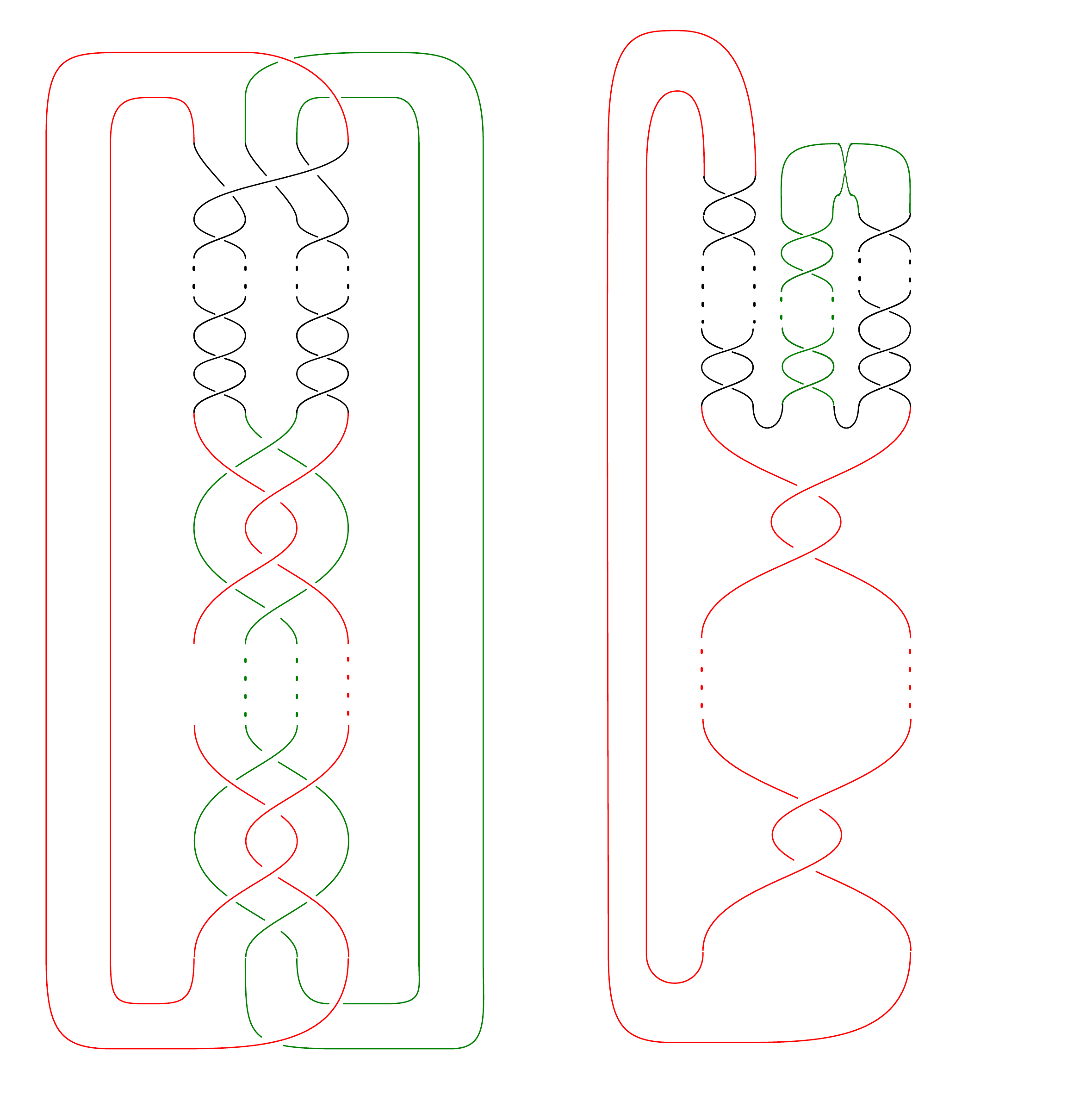
\end{figure}

The case where $n$ is odd is entirely analogous. In the braid closure of the right hand braid of Figure~\ref{both braids}, we can flip the red strands behind everything else. This also resolves the apparent asymmetry in the top of the braid we have started with.

\end{proof}
\section{Proof of the main result}
Theorem~\ref{thm:main} is an immediate consequence of Propositions~\ref{lower bound}, \ref{abe-kishimoto}, and~\ref{upper bound}. The reformulation that $\alt(K)=\lfloor \frac{1}{3}g(K)\rfloor$ is an easy computation that follows from the formula of the genus of a torus knot, given by
\[
	g(T_{p,q}) = \frac{(p-1)(q-1)}{2} \, ,
\]
for $p, q > 1$ coprime integers.

\section{Perspectives}
It is natural to wonder what the alternating numbers for torus knots of higher braid index are. Even the asymptotic behavior is unclear. To make this precise we set
\[
a_p =\lim_{n\to \infty}\frac{\alt(T_{p,i+ np})}{n}\]
for $p \geq 2, 0 \leq i < p$.
In fact, it is clear that $\limsup \alt(T_{p,i+np})/n$ exists. However, it follows from \cite{Feller_14_GordianAdjacency} that one has
\[
	\left| \alt(T_{p,k}) - \alt(T_{p,l})	\right| \leq \frac{p-1}{2} | k-l | \, ,
\]
showing that the above limit exists and that it is independent of $i$. Motivated by our geometric construction and the nature of the lower bounds, we expect that
$n \mapsto \alt(T_{p,i+ np})$ is an affine function for each $i$, and so $a_p$ would be the slope of this function, or, equivalently, the number of crossing changes needed for each additional `full twist'.

In this setup, Kanenobu's lower bound~\cite{Kanenobu_10_UpperBoundForAltOfTorusKnots}, which he obtained using Abe's lower bound~\cite{Abe_09_EstOfAltOfTorusKnot} and Gordon, Litherland and Murasugi's signature calculation~\cite{GLM}, yields
\begin{equation}\label{eq:lowerboundasympt}
\begin{split}
\frac{(p-1)^2}{4} & \leq a_p \ \ \mbox{for $p$ odd, and } \\
\frac{(p-2)p}{4} & \leq a_p \ \ \mbox{for $p$ even. }
\end{split}
\end{equation}
In fact, using the $|\tau+\upsilon|$-bound from Section~\ref{sec:lowerbounds}, one can recover~\eqref{eq:lowerboundasympt}. In particular, using the $\upsilon$-invariant, one does not get a better asymptotic lower bound than Abe's bound using the signature and the $\tau$-invariant.

Kanenobu's upper bound on the alternating number of torus knots of braid index 3 (compare Propostion~\ref{abe-kishimoto}) shows that \eqref{eq:lowerboundasympt} is an equality for $p\leq3$ and our main result Theorem~\ref{thm:main} shows that \eqref{eq:lowerboundasympt} is an equality for $p=4$ as well. The values $a_p$ for $p\geq 5$ seem out of reach at the moment.
However, maybe the geometrically constructed upper bounds generalize such that in the future the following question can be answered in the positive.
\begin{question}
Is \eqref{eq:lowerboundasympt} an equality for all positive integers $p$?
\end{question}
As a further hint in this direction, we notice that the lower bound in \eqref{eq:lowerboundasympt}, for $p$ even, is equal to the number of `band crossings' in a full twist for a suitable generalization of Figure \ref{crossingchanges_fulltwist}.

\bibliographystyle{siam}
\def\cprime{$'$}

\end{document}